\documentclass[letterpaper, 10 pt, conference]{ieeeconf}
\IEEEoverridecommandlockouts
\usepackage{cite}
\usepackage{amsmath,amssymb,amsfonts}
\usepackage{algorithmic}
\usepackage{graphicx}
\usepackage{textcomp}
\usepackage{xcolor}
\usepackage{amsthm}

\def\BibTeX{{\rm B\kern-.05em{\sc i\kern-.025em b}\kern-.08em
    T\kern-.1667em\lower.7ex\hbox{E}\kern-.125emX}}
\newtheorem{theorem}{Theorem}[]
\newtheorem{lemma}[theorem]{Lemma}
\newtheorem{proposition}[theorem]{Proposition}
\newtheorem{corollary}[theorem]{Corollary}
\theoremstyle{definition}
\newtheorem*{remark}{Remark}
    
\begin{document}

\title{The Weighted Markov-Dubins Problem\\
\thanks{$^{1}$ Deepak Prakash Kumar and Swaroop Darbha are with the Department of Mechanical Engineering, Texas A\&M University, College Station, TX 77843, USA (e-mail: {\tt\small deepakprakash1997@gmail.com, dswaroop@tamu.edu}).\\
$^{2}$ Satyanarayana Gupta Manyam is with the Infoscitex Corp., 4027 Col Glenn Hwy, Dayton, OH 45431, USA (e-mail: {\tt\small msngupta@gmail.com})\\
$^{3}$ David Casbeer is with the Autonomous Control Branch, Air Force Research Laboratory, Wright-Patterson Air Force Base, OH 45433 USA (e-mail:
{\tt\small david.casbeer@afresearchlab.com})\\
Distribution Statement A. Approved for public release, distribution unlimited. Case Number: AFRL-2022-3975.}}

\author{Deepak Prakash Kumar$^{1}$, Swaroop Darbha$^{1}$, Satyanarayana Gupta Manyam$^{2}$, David Casbeer$^{3}$}

% \and
% \and
% \and
% \and
% \and

\maketitle

\begin{abstract}
In this article, a variation of the classical Markov-Dubins problem is considered, which deals with curvature-constrained least-cost paths in a plane with prescribed initial and final configurations, different bounds for the sinistral and dextral curvatures, and penalties $\mu_L$ and $\mu_R$ for the sinistral and dextral turns, respectively. The addressed problem generalizes the classical Markov-Dubins problem and the asymmetric sinistral/dextral Markov-Dubins problem.
% The classical Markov-Dubins problem addresses curvature-constrained shortest paths in a plane for given initial and final configurations. A variation of this problem, the asymmetric sinistral/dextral Markov Dubins, considers different lower and upper bounds for the curvature. In this study, a variation to these two problems is considered, in which 
% This study considers a variation of the classical Markov-Dubins problem, which addresses curvature-constrained least cost paths in a plane for given initial and final configurations, is considered. In this study, in addition to considering different lower and upper bounds for the curvature, penalties $\mu_L, \mu_R$ are considered to taking a sinistral and dextral turn, respectively. 
The proposed formulation can be used to model an Unmanned Aerial Vehicle (UAV) with a penalty associated with a turn due to a loss in altitude while turning or a UAV with different costs for the sinistral and dextral turns due to hardware failures or environmental conditions. Using optimal control theory, the main result of this paper shows that the optimal path belongs to a set of at most $21$ candidate paths, each comprising of at most five segments. Unlike in the classical Markov-Dubins problem, the $CCC$ path, which is a candidate path for the classical Markov-Dubins problem, is not optimal for the weighted Markov-Dubins problem. Moreover, the obtained list of candidate paths for the weighted Markov-Dubins problem reduces to the standard $CSC$ and $CCC$ paths and the corresponding degenerate paths when $\mu_L$ and $\mu_R$ approach zero.
\end{abstract}

% \begin{IEEEkeywords} Markov-Dubins, Optimal control, Path planning, Unmanned vehicles
% \end{IEEEkeywords}

\section{Introduction}

Autonomous vehicles, which includes UAVs, have increasing military and civilian applications, thereby increasing the need for path planning for various scenarios that the vehicles might encounter. Path planning for UAVs is typically addressed by modeling the vehicle as a Dubins vehicle, wherein the vehicle travels forward at a constant speed and has a constraint on the curvature of the path. The classical Markov-Dubins problem, which addresses the curvature-constrained planar paths of shortest length connecting given initial and final configurations, was solved in \cite{Dubins}. For this problem, the optimal paths were obtained as $CSC$ and $CCC$ paths and degenerate paths of the same. Here, $C \in\{ L, R\}$ represents a sinistral ($L$) or a dextral ($R$) turn, respectively, of minimum turning radius, and $S$ represents a straight line segment. A variation of the Markov-Dubins problem was studied in \cite{Reeds_Shepp}, wherein the vehicle was allowed to move forwards or backward. Such a vehicle is referred to as a Reeds-Shepp vehicle.

Results in \cite{Dubins} and \cite{Reeds_Shepp} were proved without resorting to Pontryagin's Minimum Principle (PMP). The Markov-Dubins problem and the path planning for the Reeds-Shepp vehicle were approached using PMP \cite{PMP} to characterize the optimal path systematically in \cite{sussman_geometric_examples} and \cite{boissonat}, respectively. Recently, in \cite{phase_portrait_kaya}, the Markov-Dubins problem was approached using both PMP and phase portraits to simplify the proofs further. Using the obtained candidate paths through one of the previously discussed approaches, the path synthesis problem is then addressed to determine the optimal path given initial and final configurations. The path synthesis problem for the Markov-Dubins problem has been addressed in \cite{Shortest_path_synthesis_Boissonat, dubins_classification}, while \cite{path_synthesis_reeds_shepp} addresses the same problem for the Reeds-Shepp vehicle.

In the literature, various extensions/variations of the Markov-Dubins problem have been studied. For example, \cite{monroy} and \cite{sussman_3D} investigate the Markov-Dubins problem on a Riemann manifold and 3D, respectively. In \cite{dubins_circle}, the shortest curvature constrained paths for a UAV to pursue a target moving on a circle is considered. The asymmetric sinistral/dextral Markov-Dubins problem is addressed in \cite{sinistral/dextral}, wherein the bound on the sinistral and dextral curvatures need not be the same. This study obtained the same set of candidate paths as the classical Markov-Dubins problem, with the difference arising in the path synthesis. While the motivation of the study was to plan optimal paths for UAV with hardware failures or environmental conditions, considering different bounds for the sinistral and dextral curvatures cannot completely capture the preference of the UAV to take a turn in a particular direction. Hence, a weighted Markov-Dubins problem is addressed in this study, wherein different bounds for the sinistral and dextral curvatures, and penalties $\mu_L$ and $\mu_R$ are considered for sinistral and dextral turns, respectively. Furthermore, for UAVs with the same bound on the sinistral and dextral curvatures, the weighted Markov-Dubins framework that penalizes the turns accounts for the loss of altitude of the vehicle during a turn. 

From this study using Pontryagin's minimum principle, a total of $21$ candidate optimal paths are obtained, wherein each path has at most five segments. Moreover, when $\mu_L = \mu_R = 0,$ the candidate paths reduce to $CSC$ and $CCC$ paths and corresponding degenerate paths. Therefore, the proposed weighted Markov-Dubins problem generalizes the classical Markov-Dubins problem and the asymmetric sinistral/dextral Markov-Dubins problem.

The rest of the article is organized as follows. In Section~II, the problem formulation for the weighted Markov-Dubins problem is presented. In Section~III, the candidate paths for the least-cost path connecting given initial and final configurations are identified using PMP and phase portraits. In Section~IV, a typical case for the classical Markov-Dubins problem is analyzed without and with penalties. Finally, the paper is concluded in Section~V.

\section{Problem Formulation}

In this paper, the minimum cost path problem for a curvature-constrained planar vehicle with different sinistral and dextral curvatures with given initial and final configurations is considered. In contrast to \cite{sinistral/dextral}, penalties $\mu_L$ and $\mu_R$ are allocated to sinistral and dextral turns, respectively. Consider a bounded control input $u (t) \in U = [-U_R, U_L]$, which controls the rate of change of the vehicle's heading angle. Therefore, the kinematic equations for the vehicle are given by
\begin{align} \label{eq: kinematic_equations}
    \dot{x} (t) = \cos{\alpha (t)}, \, \dot{y} (t) = \sin{\alpha (t)}, \, \dot{\alpha} (t) = u (t),
\end{align}
where $x, y$ are the Cartesian coordinates for the vehicle, and $\alpha$ denotes the heading angle of the vehicle. It should be noted here that if $u > 0,$ the vehicle takes a sinistral (left) turn and if $u < 0,$ the vehicle takes a dextral (right) turn.

Equation~\ref{eq: kinematic_equations} can alternatively be written in terms of two control inputs $u_L (t) \in [0, U_L]$ and $u_R (t) \in [0, U_R],$ which denote the rate of change of the vehicle's heading angle corresponding to left and right turns, respectively.
% , as
% \begin{align}
% \begin{split}
%     \dot{x} (t) &= \cos{\alpha (t)}, \\
%     \dot{y} (t) &= \sin{\alpha (t)}, \\
%     \dot{\alpha} (t) &= u_L (t) - u_R (t),\\
%     u_L (t) u_R (t) &= 0,
% \end{split}
% \end{align}
% where the final constraint is imposed to ensure that both control inputs are not active at the same time instant.
Consider weights $\mu_L, \mu_R$ that penalize the left and right turns, respectively, such that $\mu_L, \mu_R \geq 0$ and at least one of the penalties is non-zero, i.e., $\mu_L+ \mu_R > 0.$ The minimum cost path problem can be formulated as
\begin{align} \label{eq: objective_functional}
    \min \int_{0}^{t_f} (1 + \mu_L u_L (t) + \mu_R u_R (t)) dt,
\end{align}
subject to
\begin{align}
    \dot{x} (t) = \cos{\alpha (t)}, \, \dot{y} (t) &= \sin{\alpha (t)}, \, \dot{\alpha} (t) = u_L (t) - u_R (t), \label{eq: kinematic_equations_updated} \\
    \left(u_L (t), u_R (t) \right) &\in [0, U_L] \times [0, U_R], \label{eq: range_controls} \\
    % u_L (t) u_R (t) &= 0, \label{eq: controls_nontrivial} \\
    x (0) = x_i, \, y (0) &= y_i, \, \alpha (0) = \alpha_i, \label{eq: initial_configuration} \\
    x (t_f) = x_f, \, y (t_f) &= y_f, \, \alpha (t_f) = \alpha_f. \label{eq: final_configuration}
\end{align}
In the above formulation, $t_f$ is the final free time, \eqref{eq: kinematic_equations_updated} represents the modified kinematics equation in \eqref{eq: kinematic_equations} with $u_L$ and $u_R$ as the control inputs, and \eqref{eq: range_controls} represents the range of the control inputs. Moreover, \eqref{eq: initial_configuration} and \eqref{eq: final_configuration} denote the boundary conditions, which are the given initial and final configurations, respectively.

% \todo{Do we need to talk about existence of a solution?}

\section{Characterization of the Optimal Paths}

In this section, the optimal path types for the presented problem formulation will be derived using Pontryagin's minimum principle (PMP).
% , which has been used in \cite{boissonat} and \cite{phase_portrait_kaya} to derive the optimal paths for the Markov-Dubins problem. 
Declaring $e, p, q,$ and $\beta$ to be the adjoint variables associated with the integrand in \eqref{eq: objective_functional} and the three kinematic constraints in \eqref{eq: kinematic_equations_updated}, respectively, the Hamiltonian $H$
% $H (\mathbf{x} (t), \boldsymbol{\psi} (t), \mathbf{u} (t))$ 
is given by
\begin{align}
\begin{split}
    H &= e (1 + \mu_L u_L (t) + \mu_R u_R (t)) + p (t) \cos{\alpha (t)} \\
    & \quad\, + q (t) \sin{\alpha (t)} + \beta (t) \left(u_L (t) - u_R (t) \right),
    % &= e + \left(\beta (t) + \mu_L e \right) u_L (t) + \left(\mu_R e - \beta (t) \right) u_R (t) \\
    % & \quad\, + \lambda \cos{\left(\alpha (t) - \phi \right)},
\end{split}
\end{align}
where the dependence of $H$ on the states $x, y,$ and $\alpha,$ the adjoint variables, and the control inputs
% $\mathbf{x},$ $\boldsymbol{\psi}$, $\mathbf{u}$
is not shown for brevity. The rate of change of adjoint variables is given by
\begin{align}
\begin{split}
    \dot{p} (t) &= -\frac{\partial H}{\partial x (t)} = 0, \quad \dot{q} (t) = -\frac{\partial H}{\partial y (t)} = 0, \\
    \dot{\beta} (t) &= -\frac{\partial H}{\partial \alpha (t)} = p(t) \sin{\alpha (t)} -q (t) \cos{\alpha (t)}.
\end{split}
\end{align}
Defining $\lambda = \sqrt{p^2 + q^2}$ and $\tan{\phi} = \frac{q}{p},$ the Hamiltonian and the equations for the adjoint variables can be written as
\begin{align}
\begin{split} \label{eq: Hamiltonian_updated_expression}
    H &= e + \left(\beta (t) + \mu_L e \right) u_L (t) + \left(\mu_R e - \beta (t) \right) u_R (t) \\
    & \quad\, + \lambda \cos{\left(\alpha (t) - \phi \right)},
\end{split} \\
    p &= \lambda \cos{\phi}, \, q = \lambda \sin{\phi}, \, \dot{\beta} (t) = \lambda \sin{\left(\alpha (t) - \phi \right)}. \label{eq: adjoint_equations}
\end{align}
It should be noted that from PMP, $e$ is a constant and $e \geq 0$. Moreover, for the problem formulation presented, the optimal control actions correspond to $H \equiv 0$ \cite{PMP_lecture_notes}.

\begin{lemma} \label{lemma: optimal_control_actions}
The optimal control actions are given by $(u_L, u_R) \in \mathcal{U} = \{(0, 0), (U_L, 0), (0, U_R) \}$.
\end{lemma}
\begin{proof}
From the Hamiltonian given in \eqref{eq: Hamiltonian_updated_expression}, and for it to be minimum, the following conditions should hold
\begin{align} \label{eq: conditions_optimal_control}
    \left(\beta (t) + \mu_L e \right) u_L (t) \leq 0, \, \left(\mu_R e - \beta (t) \right) u_R (t) \leq 0.
\end{align}
Further, either of the two cases holds:
\begin{itemize}
    \item $\frac{\partial H}{\partial u_L} = \beta (t) + \mu_L e \equiv 0$ or $\frac{\partial H}{\partial u_R} = \mu_R e - \beta (t) \equiv 0$. This implies that $\beta$ is a constant, and therefore, $\dot{\beta} (t) \equiv 0.$ Consequently, $\lambda \sin{\left(\alpha (t) - \phi \right)} \equiv 0.$ If $\lambda = 0,$ the Hamiltonian in \eqref{eq: Hamiltonian_updated_expression} reduces to
    \begin{align}
        H = \begin{cases}
            \left(1 + \left(\mu_L + \mu_R \right) u_R (t) \right) e, & \beta (t) \equiv -\mu_L e \\
            \left(1 + \left(\mu_L + \mu_R \right) u_L (t) \right) e, & \beta (t) \equiv \mu_R e
        \end{cases}.
    \end{align}
    Since $u_L, u_R \geq 0,$ and $\mu_L, \mu_R \geq 0$ with $\mu_L + \mu_R > 0$, $e$ should be equal to zero as $H \equiv 0$. Therefore, $\beta (t) \equiv 0.$ Hence, $(e, p, q, \beta) \equiv \mathbf{0},$ which violates the non-triviality condition of PMP. Therefore, $\dot{\beta} (t) \equiv 0 \implies \sin{\left(\alpha (t) - \phi \right)} \equiv 0$.
    % Since all the adjoint variables would then be identically equal to zero, which is not allowed by the nontriviality condition of PMP, $\dot{\beta} \equiv 0$ implies that $\sin{\left(\alpha (t) - \phi \right)} \equiv 0$. 
    Therefore, $\alpha (t) \equiv \phi$ or $\alpha (t) \equiv \phi + \pi.$ Hence, the path is a straight line segment.
    \item $\frac{\partial H}{\partial u_L} \not \equiv 0$ and $\frac{\partial H}{\partial u_R} \not \equiv 0$. Three types of control actions arise depending on the value of $\beta$ using \eqref{eq: conditions_optimal_control}:
    \begin{itemize}
        \item If $\beta (t) < -\mu_L e$, $u_L = U_L = \frac{1}{r_L}$ and $u_R = 0$. The path is an arc of a circle corresponding to a left turn of radius $r_L$.
        \item If $-\mu_L e < \beta (t) < \mu_R e,$ $u_L = 0$ and $u_R = 0.$ The path corresponds to a straight line segment.
        \item If $\beta (t) > \mu_R e$, $u_L = U_L = 0$ and $u_R = U_R = \frac{1}{r_R}.$ The path is an arc of a circle corresponding to a right turn of radius $r_R$.
    \end{itemize}
\end{itemize}
\end{proof}

\begin{proposition}
Any optimal path is a concatenation of arcs of circles of radius $r_L$ corresponding to a left turn and $r_R$ corresponding to a right turn, and straight line segments.
% parallel to a fixed direction $\phi$.
\end{proposition}

Henceforth, a left turn segment will be denoted by $L$, a right turn segment will be denoted by $R$, and a straight line segment by $S$. Moreover, a segment corresponding to a turn will be denoted by $C$.

From Lemma~\ref{lemma: optimal_control_actions}, the control actions $u_L$ and $u_R$ depend on the value of $\beta$, and can be depicted as shown in Fig.~\ref{fig: beta_line_segment_control_action}. Using Lemma~\ref{lemma: optimal_control_actions} and Fig.~\ref{fig: beta_line_segment_control_action}, and noting that $\beta$ is continuous, the observations made are given in the following proposition and lemma.

\begin{figure}[htb!]
    \centering
    \includegraphics[width = 0.8\linewidth]{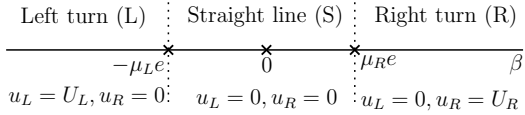}
    \caption{Control action depending on the value of $\beta$}
    \label{fig: beta_line_segment_control_action}
\end{figure}

\begin{proposition} \label{prop: inflection_points_CS_SC}
The value of $\beta$ corresponding to the inflection point of a $CS$ or $SC$ subpath of an optimal path is as follows:
\begin{enumerate}
    \item At the inflection point of an $LS$ or $SL$ subpath, $\beta = -\mu_L e$.
    \item At the inflection point of an $RS$ or $SR$ subpath, $\beta = \mu_R e$.
\end{enumerate}
\end{proposition}

\begin{lemma} \label{lemma: CC_inflection_point}
An optimal path can contain a $LR$ or $RL$ subpath if $e = 0$ and at the inflection point, $\beta = 0$.
\end{lemma}
\begin{proof}
Let $t$ be the time instant corresponding to the inflection point of an $LR$ subpath. 
% Therefore, the Hamiltonian at $t^-$ and $t^+$ can be written as
% \begin{align}
%     H (t^-) &= e + \left(\beta (t^-) + \mu_L e \right) U_L + \lambda \cos{\left(\alpha (t^-) - \phi \right)} = 0, \\
%     H (t^+) &= e + \left(\mu_R e - \beta (t^+) \right) U_R + \lambda \cos{\left(\alpha (t^+) - \phi \right)} = 0.
% \end{align}
% Since an $L$ path is a part of an optimal path, $\beta (t^-) < -\mu_L e$ from the previous lemma \textbf{Give reference to equation in the proof}. Since an $R$ path is a part of an optimal path, $\beta (t^+) > \mu_R e.$ Since $\mu_L, \mu_R > 0$ and $\beta (t)$ is continuous, $e = 0$. Moreover, this implies that at the inflection point, $\beta = 0$.
From Lemma~\ref{lemma: optimal_control_actions}, $\beta (t^-) < -\mu_L e$ and $\beta (t^+) > \mu_R e.$ Two cases arise depending on the value of $e:$ $e > 0$ and $e = 0.$ However, $e > 0 \implies \beta$ is discontinuous at $t$ as $\mu_L, \mu_R \geq 0$ and $\mu_L + \mu_R > 0.$ Since $\beta$ is continuous, $e = 0$, which implies that $\beta = 0$ at the inflection point. A similar argument can be made for an $RL$ subpath.
\end{proof}

The candidate paths for an optimal path will be obtained using a phase portrait approach. Among the adjoint variables, $\beta$ is the only time-varying adjoint variable. Moreover, the control actions, and therefore, the segments of an optimal path depend on the trajectory of $\beta.$ Hence, the phase portrait of $\beta$ can be used to determine the candidate paths. To this end, an equation relating $\beta$ and $\dot{\beta}$ is first obtained. To obtain this relation, we use Heaviside function, $\theta$, defined below: $$\theta(\eta) = \begin{cases} 1, & \eta \ge 0, \\
0, & \eta <0. 
\end{cases}
$$
From \eqref{eq: Hamiltonian_updated_expression}, noting that $H \equiv 0,$ $\lambda \cos{\left(\alpha (t) - \phi \right)}$ can be obtained as
\begin{align}
\begin{split}
    -\lambda \cos{\left(\alpha (t) - \phi \right)} &= e + \left(\beta (t) + \mu_L e \right) U_L \theta \left(-\beta (t) - \mu_L e \right) \\
    & \quad\, + \left(\mu_R e - \beta (t) \right) U_R \theta \left(\beta (t) - \mu_R e \right).
\end{split}
\label{eq:eqn13}
\end{align}
Using \eqref{eq:eqn13} and the expression for $\dot{\beta}$ from \eqref{eq: adjoint_equations}, $\beta$ and $\dot{\beta}$ are related by
\begin{align} \label{eq: beta_betadot_equation}
\begin{split}
    &\Big(e + \left(\beta (t) + \mu_L e \right) U_L \theta\left(-\beta (t) - \mu_L e \right) \\
    & \,\,+ \left(\mu_R e - \beta (t) \right) U_R \theta\left(\beta (t) - \mu_R e \right) \Big)^2 + \dot{\beta}^2 (t) = \lambda^2.
\end{split}
\end{align}
In the above equation, the parameters $r_L, r_R, \mu_L,$ and $\mu_R$ correspond to the vehicle; the relation between $\beta$ and $\dot{\beta}$ depends on $e$ and $\lambda.$ As $e$ and $\lambda$ are constants for a given path, the complete list of candidate paths can be obtained from the following cases:
\begin{itemize}
    \item Case 1: $e = 0.$ The solutions obtained are abnormal solutions as they are independent of the objective functional given in \eqref{eq: objective_functional}.
    \item Case 2: $e > 0$. In this case, $e$ can be set to one without loss of generality, as it leads to a corresponding scaling of $\lambda,$ $\beta,$ and consequently, $\dot{\beta}$. Three sub-cases arise depending on the value of $\lambda$:
    \begin{itemize}
        \item Case 2.1: $\lambda < e = 1.$
        \item Case 2.2: $\lambda = e = 1.$
        \item Case 2.3: $\lambda > e = 1.$
    \end{itemize}
\end{itemize}
The obtained solutions for each of these cases are derived in the following subsections.

\subsection{Candidate paths for $e = 0$}

When $e = 0,$ \eqref{eq: beta_betadot_equation} reduces to
\begin{align} \label{eq: beta_betadot_e_zero}
    \left(\beta (t) U_L \theta \left(-\beta (t) \right) - \beta (t) U_R \theta \left(\beta(t) \right) \right)^2 + \dot{\beta}^2 (t) = \lambda^2.
\end{align}

\begin{lemma} \label{lemma: e_0_lambda_greater_than_zero}
If $e = 0,$ then $\lambda > 0.$
\end{lemma}
\begin{proof}
If $e = 0,$ two cases arise depending on the value of $\lambda:$ $\lambda = 0$ and $\lambda > 0.$ If $\lambda = 0,$ from \eqref{eq: beta_betadot_e_zero},
\begin{align*}
    \dot{\beta} (t) \equiv 0, \quad \beta (t) U_L \theta \left(-\beta (t) \right) - \beta (t) U_R \theta\left(\beta(t) \right) \equiv 0.
\end{align*}
The only solution for $\beta$ that satisfies the above two conditions is $\beta (t) \equiv 0$ as $U_L, U_R > 0.$ However, the nontriviality condition of PMP is violated as all adjoint variables are identically equal to zero. Therefore, $e = 0 \implies \lambda > 0.$
\end{proof}

Equation~\eqref{eq: beta_betadot_e_zero} can be modified into three equations depending on $\beta$ as
\begin{align} \label{eq: equations_phase_portrait_e_0}
    f \left(\beta (t), \dot{\beta} (t) \right) = \begin{cases}
        \left(\frac{\beta (t)}{r_L} \right)^2 + \dot{\beta}^2 (t) = \lambda^2, & \beta (t) < 0 \\
        \dot{\beta}^2 (t) = \lambda^2, & \beta (t) = 0 \\
        \left(\frac{\beta (t)}{r_R} \right)^2 + \dot{\beta}^2 (t) = \lambda^2, & \beta (t) > 0
    \end{cases},
\end{align}
since $U_L = \frac{1}{r_L}$ and $U_R = \frac{1}{r_R}.$ The phase portrait obtained for $\beta$ using the above defined function is shown in Fig.~\ref{fig: phase_portrait_e_0}. Since $e = 0,$ from the proof of Lemma~\ref{lemma: optimal_control_actions} and Fig.~\ref{fig: beta_line_segment_control_action}, an $S$ segment exists over a time interval $\mathcal{I}$ of the path only when $\beta (t) = 0$ for $t \in \mathcal{I}$. Since $\dot{\beta} (t) = \pm \lambda$ when $\beta (t) = 0$ from \eqref{eq: equations_phase_portrait_e_0}, an $S$ segment is not a part of an optimal path for $e = 0$. Hence, $(\beta (t), \dot{\beta} (t)) \in \{(0, \lambda), (0, -\lambda)\}$ is an inflection point between two $C$ segments. Therefore, for $e = 0,$ candidate optimal paths are concatenations of $C$ segments.

\begin{figure}[htb!]
    \centering
    \includegraphics[width = 0.5\linewidth]{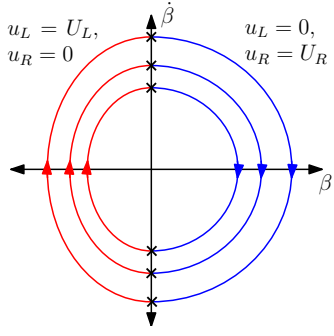}
    \caption{Phase portrait of $\beta$ for $e = 0$}
    \label{fig: phase_portrait_e_0}
\end{figure}

\begin{remark}
A non-trivial path is defined as a path wherein all subpaths have a non-zero length.
\end{remark}

\begin{theorem} \label{theorem: CCC_nonoptimality}
For $e = 0,$ a non-trivial $CCC$ path is non-optimal.
\end{theorem}
\begin{proof}
Let a non-trivial $LRL$ path be optimal. Let the time instants corresponding to the first and second inflection points be $t_1$ and $t_2,$ respectively. From Figs.~\ref{fig: beta_line_segment_control_action} and \ref{fig: phase_portrait_e_0}, $(\beta (t_1), \dot{\beta} (t_1)) = (0, \lambda)$ and $(\beta (t_2), \dot{\beta} (t_2)) = (0, -\lambda)$, since $\beta (t) < 0$ for $t \in (0, t_1) \bigcup (t_2, t_f)$ and $\beta (t) > 0$ for $t \in (t_1, t_2)$. As $\dot{\beta} (t) = \lambda \sin{\left(\alpha (t) - \phi \right)},$ and $\lambda > 0$ from Lemma~\ref{lemma: e_0_lambda_greater_than_zero},
\begin{align*}
    \sin{\left(\alpha (t_1) - \phi \right)} = 1, \quad \sin{\left(\alpha (t_2) - \phi \right)} = -1.
\end{align*}
Hence, the heading angles $\alpha (t_1)$ and $\alpha (t_2)$ are separated by an odd multiple of $\pi$. Therefore, for a non-trivial optimal $LRL$ path, the middle arc angle equals $\pi$. A similar argument can be made for a non-trivial optimal $RLR$ path. Hence, a non-trivial $CCC$ path cannot be optimal if a non-trivial $C C_\pi C$ path is shown to be non-optimal.

Consider an $L_\alpha R_\pi L_\gamma$ path, where $\alpha, \gamma > 0$. Without loss of generality, let the circle corresponding to the $L_\alpha$ segment be located at the origin, as shown in Fig.~\ref{fig: non_optimality_CCC}. Since the middle arc angle equals $\pi,$ the two inflection points $i_1$ and $i_2$ lie on the line segment connecting the centers of the circles corresponding to the two $L$ segments. Therefore, without loss of generality, the center of the circle corresponding to the $L_\gamma$ segment is placed along the $x$-axis. Consider an $L_\delta R_\pi L_\delta$ subpath of the $L_\alpha R_\pi L_\gamma$ path, where $0 < \delta < \min{\left(\alpha, \gamma, \frac{\pi}{2} \right)}$. It is claimed that there exists an alternate $RSR$ path that has a lower cost than the $L_\delta R_\pi L_\delta$ subpath. The $L_\delta R_\pi L_\delta$ subpath and the proposed $RSR$ path are shown in Fig.~\ref{fig: non_optimality_CCC} using dash-dotted and dashed lines, respectively. 
% The coordinates of the centers of the circles corresponding to the two $R$ segments of the $RSR$ path, and the heading angles  

\begin{figure}[htb!]
    \centering
    \includegraphics[width = 0.9\linewidth]{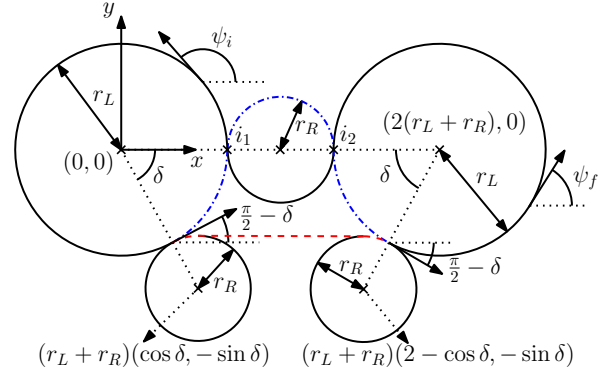}
    \caption{Non-optimality of an $L_\alpha R_\pi L_\gamma$ path}
    \label{fig: non_optimality_CCC}
\end{figure}

From Fig.~\ref{fig: non_optimality_CCC}, it can be observed that the length of the line segment ($l_S$) of the $RSR$ path equals the distance between the centers of the circles corresponding to the $R$ segments. The length $l_S$ can be obtained as $2 (r_L + r_R) (1 - \cos{\delta})$. Moreover, from this figure, it can be observed that the angle of turn for each $R$ segment equals $\frac{\pi}{2} - \delta,$ since the $S$ segment of the $RSR$ path is parallel to the $x-$axis. 
% Hence, the $R_{\frac{\pi}{2} - \delta} S_{l_S} R_{\frac{\pi}{2} - \delta}$ can replace the $L_\delta R_\pi L_\delta$ subpath of the $L_\alpha R_\delta L_\gamma$ path. 
The cost difference between the $L_\delta R_\pi L_\delta$ subpath and the $R_{\frac{\pi}{2} - \delta} S_{l_S} R_{\frac{\pi}{2} - \delta}$ path is given by
\begin{align} \label{eq: cost_difference_CCC}
\begin{split}
     \Delta c &= c_{L_\delta R_\pi L_\delta} - c_{R_{\frac{\pi}{2} - \delta} S_{l_S} R_{\frac{\pi}{2} - \delta}} \\
    &= \left(2 c_L \delta + c_R \pi \right) - \left(2 c_R \left(\frac{\pi}{2} - \delta \right) + c_S l_S \right) \\
    &= 2 \left(c_L + c_R \right) \delta - 2 c_S (r_L + r_R) (1 - \cos{\delta}),
\end{split}
\end{align}
where $c_L, c_R,$ and $c_S$ are costs associated with the $L$, $R$, and $S$ segments, respectively. The derivation of these costs follows.

Using \eqref{eq: objective_functional}, the total cost $C_L$ of a left turn as a function of the angle of the turn ($\phi_L$) can be obtained by setting $u_L = U_L = \frac{1}{r_L}$ and $u_R = 0$ as
\begin{align}
\begin{split}
    C_L &= \int_{0}^{t_f} \left(1 + \mu_L U_L \right) dt \\
    &= \left(1 + \mu_L U_L \right) t_f = \left(1 + \frac{\mu_L}{r_L} \right) \phi_L r_L = c_L \phi_L,
\end{split}
\end{align}
as the vehicle is considered to move at a unit speed. Therefore, $c_L = r_L + \mu_L$. Similarly, $c_R$ and $c_S$ can be obtained as $r_R + \mu_R$ and $1,$ respectively. Substituting the obtained expressions for $c_L, c_R,$ and $c_S$ in \eqref{eq: cost_difference_CCC}, $\Delta c$ can be simplified as
\begin{align}
\begin{split}
    \Delta c &= 2 (r_L + \mu_L + r_R + \mu_R) \delta - 2 (r_L + r_R) (1 - \cos{\delta}) \\
    &> 2 (r_L + r_R) \delta - 2 (r_L + r_R) (1 - \cos{\delta}) \\
    &= 2 (r_L + r_R) \left(\delta + \cos{\delta} - 1 \right),
\end{split}
\end{align}
as $\mu_L, \mu_R \geq 0$ with $\mu_L + \mu_R > 0$.

Consider the function $g (\delta) := \delta + \cos{\delta} - 1.$ This function is positive $\forall \delta \in \left(0, \frac{\pi}{2} \right)$ as
\begin{enumerate}
    \item $g(0) = 0$, and
    \item $\frac{d g}{d \delta} = 1 - \sin{\delta} > 0 \,\forall\, \delta \in [0, \frac{\pi}{2}),$ which implies that $g$ is strictly increasing in this range.
\end{enumerate}
Therefore, $\Delta c > 0 \,\forall\, \delta \in \left(0, \frac{\pi}{2} \right).$ Hence, the $L_\delta R_\pi L_\delta$ subpath of the $L_\alpha R_\pi L_\gamma$ path can be replaced with an $R_{\frac{\pi}{2} - \delta} S_{l_S} R_{\frac{\pi}{2} - \delta}$ path at a lower cost. Therefore, the $L_\alpha R_\pi L_\gamma$ path is not optimal. A similar argument can be made for an $R_\alpha L_\pi R_\gamma$ path. Hence, for $e = 0$, a $CCC$ path is non-optimal.
\end{proof}
\begin{proposition}
For $e = 0,$ the optimal path candidates are $C$ and $CC$, wherein the angle of each segment is at most $\pi.$
\end{proposition}

\subsection{Candidate paths for $\lambda < e = 1$}

When $e = 1,$ \eqref{eq: beta_betadot_equation} can be modified into three equations depending on $\beta$ as
\begin{align} \label{eq: equations_phase_portrait_e_1}
\begin{split}
    &f \left(\beta (t), \dot{\beta} (t) \right) \\
    &= \begin{cases}
    \left(\frac{r_L + \mu_L}{r_L} + \frac{\beta (t)}{r_L} \right)^2 + \dot{\beta}^2 (t) = \lambda^2, & \beta (t) < -\mu_L \\
    \dot{\beta}^2 (t) = \lambda^2 - 1, & -\mu_L \leq \beta (t) \leq \mu_R \\
    \left(\frac{r_R + \mu_R}{r_R} - \frac{\beta (t)}{r_R} \right)^2 + \dot{\beta}^2 (t) = \lambda^2, & \beta (t) > \mu_R
    \end{cases}.
\end{split}
\end{align}
If $\lambda < 1,$ $\dot{\beta} (t)$ does not have a real solution for $-\mu_L \leq \beta \leq \mu_R$; 
% Therefore, the dynamics of $\beta$ is not well defined in this region. 
hence, the initial condition for $\beta$ cannot lie in $[-\mu_L, \mu_R]$. The following lemma will also establish that for any other initial condition, $\beta (t) \notin [-\mu_L, \mu_R]$ for any $t > 0$.
\begin{lemma} \label{lemma: beta_justification_e_1_lambda_less_than_1}
For $\lambda < e = 1,$ if $\beta (0) < -\mu_L,$ then $\beta (t) < -\mu_L, \,\forall\, t > 0$. If $\beta (0) > \mu_R,$ then $\beta (t) > \mu_R, \,\forall\, t > 0.$
\end{lemma}
\begin{proof}
Suppose $\beta(0) + \mu_L <0$; should $\beta(t)+\mu_L = 0$ for any $t$, equation \eqref{eq: equations_phase_portrait_e_1} implies that  $1 \leq 1 + \dot \beta^2 (t) = \lambda^2 <1$, a contradiction; hence, $\beta(t)+\mu_L <0$ for all $t$ by continuity of $\beta$. A similar reasoning can be used to show that $\beta(0) -\mu_R >0 \implies \beta(t)-\mu_R >0$ for all $t>0$.
% Let $\beta (0) < -\mu_L.$ The dynamics of $\beta$ will be described initially by the equation corresponding to $\beta (t) < -\mu_L$ in \eqref{eq: equations_phase_portrait_e_1}, which corresponds to an ellipse whose center is on the $\beta-$axis. The intersection points of this ellipse with the $\beta-$axis are at $\beta = (\pm \lambda - 1) r_L - \mu_L.$ Since $\lambda < 1,$ $(\pm \lambda - 1) < 0.$ Hence, $(\pm \lambda - 1) r_L - \mu_L < -\mu_L,$ since $r_L > 0.$ Therefore, $\beta (0) < -\mu_L \implies \beta (t) < -\mu_L \,\forall\, t > 0.$
% Similarly, if $\beta (0) > \mu_R,$ the dynamics of $\beta$ will be described initially by the equation corresponding to $\beta (t) > \mu_R$ in \eqref{eq: equations_phase_portrait_e_1}, which corresponds to an ellipse whose center is on the $\beta-$axis. The intersection points of this ellipse with the $\beta-$axis are at $\beta = (1 \pm \lambda) r_R + \mu_R$, which lie to the right of the line $\beta = \mu_R$ as $\lambda < 1.$ Hence, $\beta (0) > \mu_R \implies \beta (t) > \mu_R \,\forall\, t > 0.$
\end{proof}

The phase portrait obtained for $\lambda < e = 1$ using \eqref{eq: equations_phase_portrait_e_1} and the previous observations is shown in Fig.~\ref{fig: phase_portrait_e_1_lambda_less_than_1}. Using Lemma~\ref{lemma: beta_justification_e_1_lambda_less_than_1} and Fig.~\ref{fig: phase_portrait_e_1_lambda_less_than_1}, it can be observed that the only candidate path is $C$ for $\lambda < e = 1$.
\begin{figure}[htb!]
    \centering
    \includegraphics[width = 0.9\linewidth]{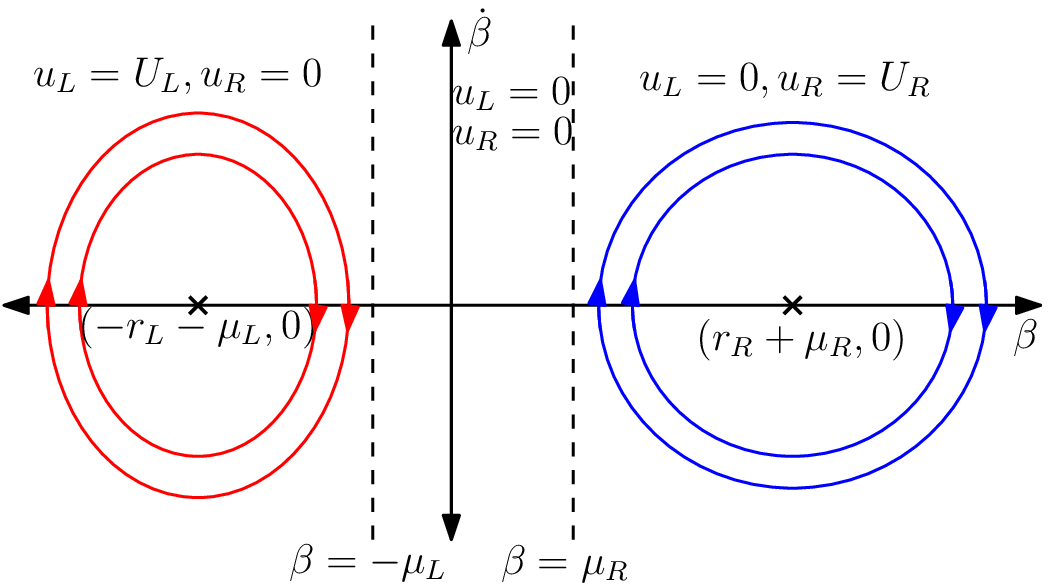}
    \caption{Phase portrait of $\beta$ for $\lambda < e = 1$}
    \label{fig: phase_portrait_e_1_lambda_less_than_1}
\end{figure}

\subsection{Candidate paths for $\lambda = e = 1$}

% In the proof of Lemma~\ref{lemma: beta_justification_e_1_lambda_less_than_1}, the intersection points of the ellipse corresponding to $\beta (t) < -\mu_L$ with the $\beta-$axis were obtained to be at $\beta = (\pm \lambda - 1) r_L - \mu_L.$
The equation corresponding to $\beta (t) < -\mu_L$ in \eqref{eq: equations_phase_portrait_e_1} correspond to an ellipse, whose origin lies on the $\beta-$axis. The intersection points of this ellipse with the $\beta-$axis are at $\beta = (\pm \lambda - 1) r_L - \mu_L.$
For $\lambda = 1,$ the intersection points are at $\beta = -\mu_L$ and $\beta = -2 r_L - \mu_L$. Hence, this ellipse is tangential to the line $\beta = -\mu_L.$ Similarly, the ellipse corresponding to $\beta (t) > \mu_R$ can be observed to be tangential to the line $\beta = \mu_R.$ It should also be noted that when $\lambda = 1,$ $\dot{\beta} (t) = 0$ for $-\mu_L \leq \beta (t) \leq \mu_R$ from \eqref{eq: equations_phase_portrait_e_1}. Using these observations and \eqref{eq: equations_phase_portrait_e_1}, the phase portrait for $\beta$ for $\lambda = e = 1$ can be obtained as shown in Fig.~\ref{fig: phase_portrait_e_1_lambda_1}.
% Moreover, the ellipses corresponding to $\beta < -\mu_L$ and $\beta > \mu_R$ are tangential to the lines $\beta = -\mu_L$ and $\beta = \mu_R,$ respectively, which can be obtained by computing the intersection points of each ellipse with the $\beta-$axis. The phase portrait obtained for $\lambda = e = 1$ using \eqref{eq: equations_phase_portrait_e_1} is shown in Fig.~\ref{fig: phase_portrait_e_1_lambda_1}.
% \begin{lemma}
% For $e = \lambda = 1,$ an optimal path can contain a straight line segment if and only if $\alpha (t) = \phi$ or $\alpha (t) = \phi + \pi$ over the straight line segment.
% \end{lemma}
% \begin{proof}
% Let the optimal path contain a straight line segment. Therefore, during the time interval corresponding to the straight line segment, $\beta (t) \in [-\mu_L, \mu_R]$. However, for $\beta (t) \in [-\mu_L, \mu_R],$ $\dot{\beta} (t) = 0$ from \eqref{eq: equations_phase_portrait_e_1}. Therefore, $\sin{\left(\alpha (t) - \phi \right)} = 0$ over the straight line segment as $\lambda \neq 0.$ Hence, $\alpha (t) = \phi$ or $\alpha (t) = \phi + \pi$ over the straight line segment.
% Let $\alpha (t) = \phi$ or $\alpha (t) = \phi + \pi.$ Hence, $\dot{\beta} (t) = \lambda \sin{\left(\alpha (t) - \phi \right)} = 0.$ Using \eqref{eq: equations_phase_portrait_e_1} and Fig.~\ref{fig: phase_portrait_e_1_lambda_1}, $\dot{\beta} = 0$ can occur only when $\beta (t) \in [-\mu_L, \mu_R],$ or $(\beta (t), 0)$ 
% \end{proof}
The candidate paths for $\lambda = e = 1$ can be obtained using the following lemma.

\begin{figure}[htb!]
    \centering
    \includegraphics[width = 0.9\linewidth]{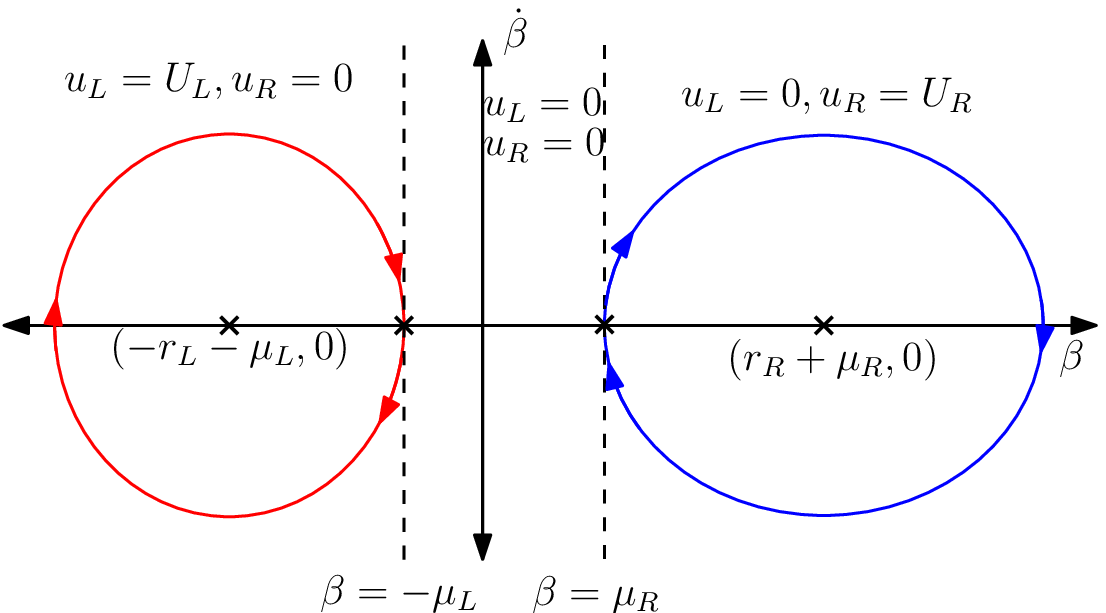}
    \caption{Phase portrait of $\beta$ for $\lambda = e = 1$}
    \label{fig: phase_portrait_e_1_lambda_1}
\end{figure}

% \begin{lemma}
% For $\lambda = e = 1,$
% \begin{itemize}
%     \item If $\beta (0) \leq -\mu_L,$ the optimal path is $LSL$ or a degenerate path of the same.
%     \item If $\beta (0) \in (-\mu_L, \mu_R),$ the optimal path is $S.$
%     \item If $\beta (0) \geq \mu_R,$ the optimal path is $RSR$ or a degenerate path of the same.
% \end{itemize}
% \end{lemma}
\begin{lemma}
For $\lambda = e = 1,$ the optimal path is $LSL,$ $RSR$, or a degenerate path of $LSL$ and $RSR$ paths.
\end{lemma}
\begin{proof}
Let $\beta (0) \leq -\mu_L.$ From \eqref{eq: equations_phase_portrait_e_1} and Fig.~\ref{fig: phase_portrait_e_1_lambda_1}, if $\beta (t) < -\mu_L,$ the evolution of $\beta$ corresponds to an ellipse that is tangential to the line $\beta = -\mu_L$ and lies to the left of $\beta = -\mu_L$, and if $\beta (t) = -\mu_L,$ $\dot{\beta} (t) = 0$ since $\lambda = e = 1.$ Therefore, $\beta (0) \leq -\mu_L \implies \beta (t) \leq -\mu_L \,\forall\, t > 0.$
% From \eqref{eq: equations_phase_portrait_e_1} and Fig.~\ref{fig: phase_portrait_e_1_lambda_1}, it can be observed that if $\beta (0) \leq -\mu_L,$ $\beta (t) \leq -\mu_L \,\forall\, t > 0.$ 
% It should be noted that if $\beta = -\mu_L,$ $\dot{\beta} = 0$ from \eqref{eq: equations_phase_portrait_e_1}. 
As $\beta (t) = -\mu_L \implies \dot{\beta} (t) = 0,$ $\beta (t)$ can be identically equal to $-\mu_L$ over a time interval. From Lemma~\ref{lemma: optimal_control_actions}, $\beta \equiv -\mu_L$ (with $e = 1$) corresponds to an $S$ segment. Since $\beta (t) < -\mu_L$ corresponds to an $L$ segment and $\beta (t) \equiv -\mu_L$ corresponds to an $S$ segment, the optimal path has alternate $L$ and $S$ segments using the previous observations.

Consider an $SLS$ path. Here, $\beta (0) = -\mu_L,$ $\beta (t_f) = -\mu_L.$ 
% As the ellipse corresponding to the left turn is tangential to the line $\beta = -\mu_L,$ the ellipse must be traversed completely. Moreover, 
Since $\dot{\beta} = \lambda \sin{\left(\alpha (t) - \phi \right)},$ and at the inflection points, $\beta = -\mu_L$ and $\dot{\beta} = 0$, the angle of the $L$ segment is a multiple of $2 \pi.$ Hence, an $SLS$ path is not optimal. Therefore, if $\beta (0) \leq -\mu_L,$ the candidate paths are $S, L, LS, SL,$ and $LSL.$

Using a similar argument for $\beta (0) \geq \mu_R,$ the candidate paths can be obtained as $S, R, RS, SR,$ and $RSR$. If $\beta (0) \in (-\mu_L, \mu_R),$ $\beta (t) = \beta (0) \,\forall\, t > 0$ as $\dot{\beta} = 0$ from \eqref{eq: equations_phase_portrait_e_1}. Using Lemma~\ref{lemma: optimal_control_actions}, the optimal path in this scenario is $S$.
\end{proof}

\subsection{Candidate paths for $\lambda > e = 1$}

As the intersection points of the ellipse corresponding to $\beta (t) < -\mu_L$ with the $\beta-$axis are at $\beta = (\pm \lambda - 1) r_L - \mu_L$, 
% which were obtained in the proof of Lemma~\ref{lemma: beta_justification_e_1_lambda_less_than_1},
for $\lambda > 1,$ one of the intersection points lies to the right of the line $\beta = -\mu_L$. Therefore, the ellipse corresponding to $\beta (t) < -\mu_L$ intersects with the line $\beta = -\mu_L$ at two points. The coordinates of these two points can be obtained from \eqref{eq: equations_phase_portrait_e_1} as $(-\mu_L, \pm \sqrt{\lambda^2 - 1})$. Similarly, the ellipse corresponding to $\beta (t) > \mu_R$ can be observed to intersect the line $\beta = \mu_R$ at two points, whose coordinates are given by $(\mu_R, \pm \sqrt{\lambda^2 - 1})$. Moreover, for $-\mu_L \leq \beta (t) \leq \mu_R,$ $\dot{\beta} (t) = \pm \sqrt{\lambda^2 - 1}$ from \eqref{eq: equations_phase_portrait_e_1}. Using these observations and \eqref{eq: equations_phase_portrait_e_1}, the phase portrait for $\beta$ for $\lambda > e = 1$ can be obtained as shown in Fig.~\ref{fig: phase_portrait_e_1_lambda_greater_than_1}. From this figure, the following proposition follows.

\begin{figure}[htb!]
    \centering
    \includegraphics[width = \linewidth]{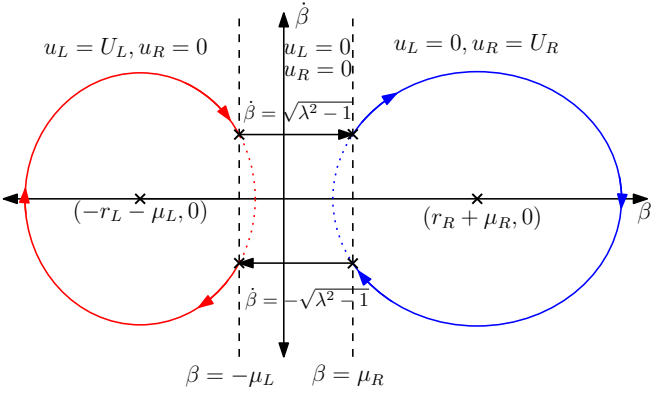}
    \caption{Phase portrait of $\beta$ for $\lambda > e = 1$}
    \label{fig: phase_portrait_e_1_lambda_greater_than_1}
\end{figure}

\begin{proposition} \label{prop: optimal_path_sequence_lambda_greater_than_1}
The segments of an optimal path for $\lambda > e = 1$ is a cyclic permutation of $L, S_1, R,$ and $S_2$ segments, wherein $S_1$ and $S_2$ denote straight line segments with $\dot{\beta} (t) = \sqrt{\lambda^2 - 1}$ and $\dot{\beta} (t) = -\sqrt{\lambda^2 - 1},$ respectively. 
\end{proposition}
\begin{corollary} \label{corr: CC_non-optimality}
For $\lambda > e = 1,$ a non-trivial $CC$ path is non-optimal.
\end{corollary}

It should be noted that henceforth, both $S_1$ and $S_2$ segments will be denoted as an $S$ segment. Using the sequence of segments in an optimal path given in Proposition~\ref{prop: optimal_path_sequence_lambda_greater_than_1} and Fig.~\ref{fig: phase_portrait_e_1_lambda_greater_than_1}, it is easy to deduce whether an $S$ segment denotes an $S_1$ or $S_2$ segment in an optimal path.

\begin{lemma} \label{lemma: angle_C_segment_lambda_greater_than_1}
For $\lambda > e = 1,$ the angle of the $C$ segment of a non-trivial optimal $SCS$ path is $2 \pi - 2 \cos^{-1}{\left(\frac{1}{\lambda} \right)}$.
\end{lemma}
\begin{proof}
Let a non-trivial $SLS$ path be optimal. Let the time instants corresponding to the first and second inflection points be $t_1$ and $t_2,$ respectively. Hence, $\beta (t_1) = \beta (t_2) = -\mu_L.$ Moreover, $\beta (t) < -\mu_L$ for $t \in (t_1, t_2)$ as an $L$ segment is traversed. From Fig.~\ref{fig: phase_portrait_e_1_lambda_greater_than_1}, $\beta$ and $\dot{\beta}$ at the two inflection points can be obtained as
\begin{align}
    (\beta (t_1), \dot{\beta} (t_1)) = \left(-\mu_L, -\sqrt{\lambda^2 - 1} \right), \label{eq: SLS_t1} \\
    (\beta (t_2), \dot{\beta} (t_2)) = \left(-\mu_L, \sqrt{\lambda^2 - 1} \right). \label{eq: SLS_t2}
\end{align}
Using the equation of the ellipse corresponding to $\beta (t) < -\mu_L$ from \eqref{eq: equations_phase_portrait_e_1}, and since $\dot{\beta} (t) = \lambda \sin{\left(\alpha (t) - \phi \right)},$ the coordinates $(\beta, \dot{\beta})$ of a point $P_L$ on the ellipse can be parameterized in terms of the heading angle $\alpha (t)$ as
\begin{align} \label{eq: point_left_ellipse}
    P_L = (-r_L -\mu_L - \lambda r_L \cos{\left(\alpha (t) - \phi \right)}, \lambda \sin{\left(\alpha (t) - \phi \right)}).
\end{align}
It should be noted that the above parameterization was chosen for $P_L (\beta, \dot{\beta})$ to ensure that differentiating $\beta (\alpha (t))$ with respect to $t$ yields the function $\dot{\beta} (\alpha (t))$.

Comparing \eqref{eq: SLS_t1} and \eqref{eq: SLS_t2} with \eqref{eq: point_left_ellipse},
\begin{align}
    \sin{\left(\alpha (t_1) - \phi \right)} = \frac{-\sqrt{\lambda^2 - 1}}{\lambda}, \,\, \cos{\left(\alpha (t_1) - \phi \right)} = \frac{-1}{\lambda}, \\
    \sin{\left(\alpha (t_2) - \phi \right)} = \frac{\sqrt{\lambda^2 - 1}}{\lambda}, \,\, \cos{\left(\alpha (t_2) - \phi \right)} = \frac{-1}{\lambda}.
\end{align}
Since $\dot{\alpha} (t) = U_L > 0$ over an $L$ segment, $\alpha (t) = \alpha (t_1) + U_L (t - t_1) > \alpha (t_1).$ Hence, the angles $\alpha (t_1) - \phi, \alpha (t_2) - \phi,$ and the angle of the $L$ segment ($\phi_L$) can be represented as shown in Fig.~\ref{fig: SLS_path_angle_left_turn}. Therefore, $\phi_L = 2 \pi - 2 \theta = 2 \pi - 2 \cos^{-1}{\left(\frac{1}{\lambda} \right)}$. Hence, $\phi_L \in (\pi, 2 \pi).$

\begin{figure}[htb!]
    \centering
    \includegraphics[width = 0.6\linewidth]{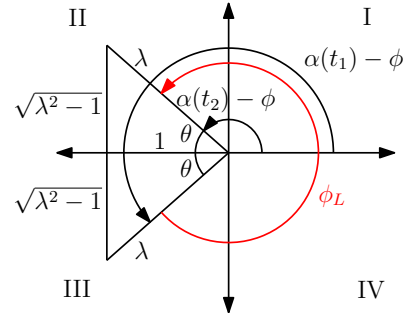}
    \caption{Angle of an $L$ segment for an optimal $SLS$ path for $\lambda > e = 1$}
    \label{fig: SLS_path_angle_left_turn}
\end{figure}

The angle of the $R$ segment of a non-trivial optimal $SRS$ path can be derived similarly. The coordinates of the first and second inflection points corresponding to time instants $t_1$ and $t_2$ on the phase portrait can be obtained as
\begin{align*}
    &(\beta (t_1), \dot{\beta} (t_1)) = \left(\mu_R, \sqrt{\lambda^2 - 1} \right), \\
    &(\beta (t_2), \dot{\beta} (t_2)) = \left(\mu_R, -\sqrt{\lambda^2 - 1} \right),
\end{align*}
since $\beta (t) > \mu_R$ for $t \in (t_1, t_2).$ A point $P_R$ on the ellipse corresponding to $\beta (t) > \mu_R$ from \eqref{eq: equations_phase_portrait_e_1} can be parameterized in terms of $\alpha (t)$ as
\begin{align*}
    P_R = (r_R + \mu_R + \lambda r_R \cos{\left(\alpha (t) - \phi \right)}, \lambda \sin{\left(\alpha (t) - \phi \right)}).
\end{align*}
Since $\dot{\alpha} (t) = -U_R < 0,$ $\alpha (t) = \alpha (t_1) - U_R (t - t_1) < \alpha (t_1).$ Using a similar construction as given in Fig.~\ref{fig: SLS_path_angle_left_turn}, the angle of the $R$ segment can be obtained as $\phi_R = 2 \pi - 2 \cos^{-1}{\left(\frac{1}{\lambda} \right)}$. Hence, $\phi_R \in (\pi, 2 \pi).$
\end{proof}

\begin{lemma} \label{lemma: length_S_segment_lambda_greater_than_1}
For $\lambda > e = 1,$ the length of the line segment of a non-trivial optimal $LSR$ or $RSL$ path is $\frac{\mu_L + \mu_R}{\sqrt{\lambda^2 - 1}}$.
\end{lemma}
\begin{proof}
Consider a non-trivial optimal $LSR$ path. Let the time instants corresponding to the first and second inflection points be $t_1$ and $t_2,$ respectively. Hence, $\beta (t_1) = -\mu_L$ and $\beta (t_2) = \mu_R$ using Proposition~\ref{prop: inflection_points_CS_SC}. Using Proposition~\ref{prop: optimal_path_sequence_lambda_greater_than_1} and Fig.~\ref{fig: phase_portrait_e_1_lambda_greater_than_1}, $\dot{\beta} (t) = \sqrt{\lambda^2 - 1}$ for $t \in [t_1, t_2]$. Since $\dot{\beta}$ is a constant, $\beta (t) = \beta (t_1) + \sqrt{\lambda^2 - 1} (t - t_1)$ for $t \in [t_1, t_2].$ As $\beta (t_2) = \mu_R,$ and since the vehicle is considered to move at a unit speed,
\begin{align} \label{eq: length_line_segment_of_subpath}
    l_S = t_2 - t_1 = \frac{\beta (t_2) - \beta (t_1)}{\sqrt{\lambda^2 - 1}} = \frac{\mu_L + \mu_R}{\sqrt{\lambda^2 - 1}},
\end{align}
where $l_S$ is the length of the $S$ segment of the $LSR$ path. The length of the $S$ segment of an $RSL$ path can be similarly obtained as the same expression.
\end{proof}

\begin{lemma}
For $\lambda > e = 1,$ an optimal path contains at most five segments.
\end{lemma}
\begin{proof}
From Proposition~\ref{prop: optimal_path_sequence_lambda_greater_than_1}, four candidate six-segment paths can be obtained, which are $LSRSLS,$ $SRSLSR,$ $RSLSRS,$ and $SLSRSL.$ Consider a non-trivial optimal $SLSRSL$ path. Using Lemmas~\ref{lemma: angle_C_segment_lambda_greater_than_1} and \ref{lemma: length_S_segment_lambda_greater_than_1}, the $LSRS$ subpath is completely parameterized using $\lambda.$ In particular, the angles of the $L$ and $R$ segments of this subpath are equal. Consider an initial and final configuration connected by this $LSRS$ subpath such that
\begin{enumerate}
    \item The center of the circle corresponding to the $L$ segment is at the origin, and
    \item The initial heading angle is $\frac{3\pi}{2}$.
\end{enumerate}
As the $L$ and $R$ segments have an equal angle, the final heading angle is also $\frac{3\pi}{2}$. Moreover, any initial and final configuration connected by the considered $LSRS$ subpath can be transformed to the above-stated representation. It should be noted that the final position in this representation depends on the parameters $r_L, r_R, \mu_L, \mu_R,$ and $\lambda.$

The considered initial position ($P_i$), the obtained final position ($P_f$), and the generated $LSRS$ path are shown in Fig.~\ref{fig: LSRS_path}. An alternate $SLSR$ path can be used to connect the same initial and final configurations such that the final $S$ segment in the $LSRS$ path is removed and inserted before the $LSR$ subpath. Such an operation can be performed since the initial and final heading angles are equal. As the $LSRS$ path is modified to obtain the $SLSR$ path without changing any segment's parameter, the $LSRS$ and $SLSR$ paths have the same cost. The obtained $SLSR$ path from the $LSRS$ path is shown in Fig.~\ref{fig: LSRS_path} using dash-dotted lines.

\begin{figure}[htb!]
    \centering
    \includegraphics[width = 0.6\linewidth]{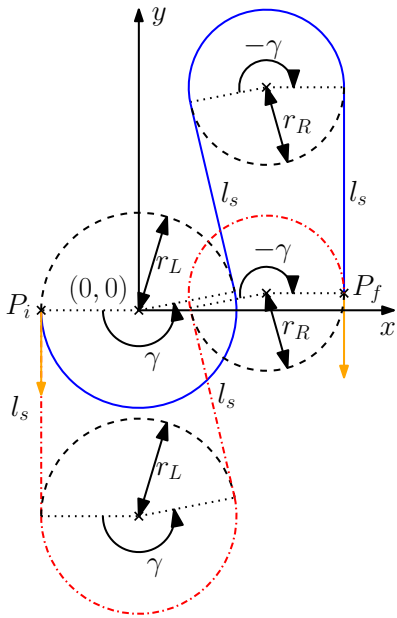}
    \caption{Configurations connected by an $LSRS$ path and an alternate $SLSR$ path parameterized using a single parameter}
    \label{fig: LSRS_path}
\end{figure}

Hence, from the initially considered non-trivial $SLSRSL$ path, an alternate non-trivial $SSLSRL$ path can be constructed to connect the same initial and final configurations at the same cost. However, using Corollary~\ref{corr: CC_non-optimality}, the non-trivial $SSLSRL$ path is non-optimal, as it contains a non-trivial $CC$ subpath. Therefore, the non-trivial $SLSRSL$ path is not optimal. By string reversal, a non-trivial $LSRSLS$ path is not optimal. Using a similar argument for a non-trivial $SRSLSR$ path, an alternate non-trivial $SSRSLR$ path can be constructed by replacing the $RSLS$ subpath with an $SRSL$ subpath. As the non-trivial $SSRSLR$ path is non-optimal using Corollary~\ref{corr: CC_non-optimality}, the non-trivial $SRSLSR$ path is not optimal. Therefore, a non-trivial $RSLSRS$ path is also not optimal. Hence, for $\lambda > e = 1,$ an optimal path contains at most five segments.
\end{proof}

\begin{lemma}
For $\lambda > e = 1,$ there exists optimal four-segment paths of the same cost as non-trivial $SLSRS$ and $SRSLS$ paths.
\end{lemma}
\begin{proof}
Consider a non-trivial optimal $SLSRS$ path. Using Lemmas~\ref{lemma: angle_C_segment_lambda_greater_than_1} and \ref{lemma: length_S_segment_lambda_greater_than_1}, the $LSR$ subpath is parameterized using $\lambda.$ Hence, the angles of the $L$ and $R$ segments are equal. Consider the $LSRS$ subpath of the $SLSRS$ path. The initial and final configurations connected by this $LSRS$ subpath can be represented similar to the representation shown in Fig.~\ref{fig: LSRS_path}. However, the length of the final $S$ segment is not parameterized using $\lambda,$ and is therefore not necessarily equal to the length of the initial $S$ segment. An alternate $SLSR$ path of the same cost connecting the same configurations can be constructed by removing the final $S$ segment in the $LSRS$ path and inserting it before the $LSR$ subpath. Therefore, from the initially considered non-trivial $SLSRS$ path, an alternate non-trivial $SSLSR$ path, which is the same as a non-trivial $SLSR$ path, can be constructed with the same cost connecting the same initial and final configurations.

Using a similar argument, from a non-trivial $SRSLS$ path, an alternate non-trivial $SSRSL$ path, which is the same as a non-trivial $SRSL$ path, can be constructed by replacing the $RSLS$ subpath with an $SRSL$ subpath. Both the $SRSLS$ path and the $SRSL$ path connect the same initial and final configurations at the same cost. Hence, the non-trivial $SLSRS$ and $SRSLS$ paths are redundant and need not be considered as candidate paths for the optimal solution.
\end{proof}

\begin{theorem}
There are at most $21$ candidate paths for the minimum cost path for the weighted Markov-Dubins problem, which are given in Table~\ref{tab: optimal_paths}.
\begin{table}[htb!]
    \centering
    \caption{List of candidate paths for the weighted Markov-Dubins problem}
    \label{tab: optimal_paths}
    \begin{tabular}{|c|c|c|}
    \hline
    \textbf{No. of segments} & \textbf{Candidate paths} & \textbf{No. of paths} \\ \hline
    $1$ & $S,$ $C$ & $3$ \\ \hline
    $2$ & $SC,$ $CS,$ $CC$ & $6$ \\ \hline
    $3$ & $CSC,$ $SCS$ & $6$ \\ \hline
    $4$ & $LSRS,$ $SRSL,$ $RSLS,$ $SLSR$ & $4$ \\ \hline
    $5$ & $LSRSL,$ $RSLSR$ & $2$ \\ \hline
    \end{tabular}
\end{table}
\end{theorem}

\begin{remark}
For $\mu_L = \mu_R = 0,$ the list of candidate paths in Table~\ref{tab: optimal_paths} reduces to paths of type $CSC,$ $CCC,$ and degenerate paths of the same. This is because the length of an intermediate $S$ segment in an optimal four or five segment path tends to zero as $\mu_L, \mu_R \rightarrow 0$ from \eqref{eq: length_line_segment_of_subpath}. Hence, the weighted Markov-Dubins problem generalizes the standard Markov-Dubins problem \cite{Dubins} and the asymmetric sinistral/dextral Markov-Dubins problem \cite{sinistral/dextral}.
\end{remark}

\section{Results}

Given the list of candidate paths in Table~\ref{tab: optimal_paths}, each candidate path should be generated (if it exists) for given initial and final configurations, vehicle parameters, and penalties. For this purpose, the approach used in \cite{dubins_classification} is adapted to derive closed-form expressions for the parameters of a given path, which are the angles for each $L$ and $R$ segment and the length of an $S$ segment. The derivation of these expressions uses a corollary of Lemmas~\ref{lemma: angle_C_segment_lambda_greater_than_1} and \ref{lemma: length_S_segment_lambda_greater_than_1}, wherein each candidate four and five-segment paths are parameterized by at most three parameters. For example, for an optimal $LSRSL$ path, consider angles $\phi_1, \phi_2,$ and $\phi_3,$ for the first $L$ segment, the $R$ segment, and the final $L$ segment, respectively. For the optimal $LSRSL$ path, $\lambda$ can be expressed as a function of $\phi_2$ using Lemma~\ref{lemma: angle_C_segment_lambda_greater_than_1}. Since the length of the $S$ segments is a function of $\lambda, \mu_L,$ and $\mu_R$, due to Lemma~\ref{lemma: length_S_segment_lambda_greater_than_1}, it can be expressed as a function of $\phi_2, \mu_L,$ and $\mu_R.$ Hence, an optimal $LSRSL$ path is parameterized using three parameters.

Consider the initial and final position of the vehicle to be at the origin and the initial and final headings to be $0^\circ$ and $180^\circ$, respectively. Let $r_L = r_R = 1$ m. For the unweighted Markov-Dubins problem, the $LRL$ path was obtained to be optimal. However, this path is no longer optimal with the introduction of weights. For $\mu_L = \mu_R = 1,$ which corresponds to a $100\%$ penalty imposed on turns, an $LSRSL$ path was obtained to be optimal. The two cases are shown in Fig.~\ref{fig: comparison_LRL_LSRSL}. For the $LRL$ path, the angles of the $L$ and $R$ segments are $60^\circ$ and $300^\circ$, respectively. With the introduction of the penalties, the angles of the $L$ and $R$ segments are reduced to $32.53^\circ$ and $245.07^\circ,$ respectively, corresponding to the $LSRSL$ path due to the $S$ segments, which are of length $1.28$ m. Since an $S$ segment is of lower cost than $L$ and $R$ segments, which can be observed from the derivation of the costs of the $L, R,$ and $S$ segments in the proof of Theorem~\ref{theorem: CCC_nonoptimality}, the resultant $LSRSL$ path is of a lower cost compared to an $LRL$ path for the weighted case.

\begin{figure}[htb!]
    \centering
    \includegraphics[width = 0.9\linewidth]{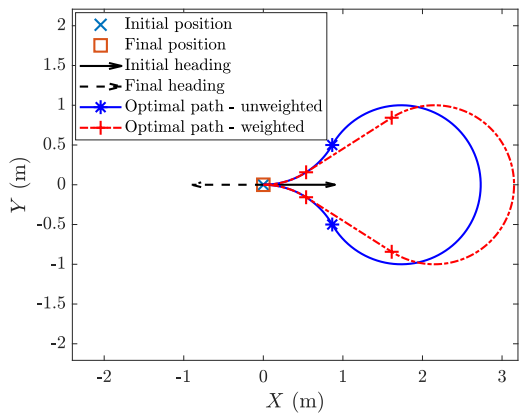}
    \caption{Comparison of optimal paths for $r_L = r_R = 1$ m without weights and with $\mu_L = \mu_R = 1$}
    \label{fig: comparison_LRL_LSRSL}
\end{figure}

% A parametric study was performed considering all the candidate paths given in Table~\ref{tab: optimal_paths} to determine if the list of candidate paths can be further reduced. For this purpose, the initial position was fixed to be at the origin, and the final position was considered to be along the $X-$axis without loss of generality. For this study, the parameters are $r_L, r_R, \mu_L, \mu_R,$ the initial heading angle $\alpha_i$ and final heading angle $\alpha_f,$ and the Euclidean distance between the initial and final positions $d.$
% The parameter $r_L$ was set to one, as all other parameters will be scaled accordingly. Moreover, $\alpha_i, \alpha_f \in [0, 2 \pi),$ $d \in [0, 20],$ $r_R \in [0.5, 3],$ $\mu_L, \mu_R \in [0.1, 5]$ was considered. The results obtained are shown in Fig.~\ref{fig: parametric_study}, wherein the number of configurations for which a path is optimal is shown for each path type. It should be noted that the one segment and two segment paths were considered to be degenerate $CSC$ paths during the path generation. From this figure, it can be observed that the $CCC$ paths are never optimal. Moreover, it can be observed that the obtained list of candidate paths in Table~\ref{tab: optimal_paths} cannot be further reduced. 

% \begin{figure}[htb!]
%     \centering
%     \includegraphics[width = 0.9\linewidth]{Parametric_study_weighted_Dubins.eps}
%     \caption{Parametric study for weighted Dubins}
%     \label{fig: parametric_study}
% \end{figure}

\section{Conclusion}

In this paper, a variant of the classical Markov-Dubins problem, referred to as \textit{weighted Markov-Dubins problem}, is addressed. The considered problem addresses curvature-constrained planar least-cost paths connecting given initial and final configurations with different bounds of the curvature for sinistral and dextral turns, and penalties $\mu_L$ and $\mu_R$ associated with those turns, respectively. The proposed problem is solved using Pontryagin's minimum principle. A total of $21$ candidate least-cost paths were obtained, each comprising of at most five segments. Moreover, when $\mu_L = \mu_R = 0,$ the candidate paths reduce to paths of type $CSC$ and $CCC$ and degenerate paths of the same. Hence, the addressed weighted Markov-Dubins problem generalizes the classical Markov-Dubins problem and the asymmetric sinistral/dextral Markov-Dubins problem.

\section*{Acknowledgment}

The authors gratefully acknowledge Christopher Montez, Texas A\&M University, for useful discussions.

\bibliographystyle{IEEEtran}
\bibliography{IEEEabrv, cite}

\begin{thebibliography}{10}
\providecommand{\url}[1]{#1}
\csname url@samestyle\endcsname
\providecommand{\newblock}{\relax}
\providecommand{\bibinfo}[2]{#2}
\providecommand{\BIBentrySTDinterwordspacing}{\spaceskip=0pt\relax}
\providecommand{\BIBentryALTinterwordstretchfactor}{4}
\providecommand{\BIBentryALTinterwordspacing}{\spaceskip=\fontdimen2\font plus
\BIBentryALTinterwordstretchfactor\fontdimen3\font minus
  \fontdimen4\font\relax}
\providecommand{\BIBforeignlanguage}[2]{{%
\expandafter\ifx\csname l@#1\endcsname\relax
\typeout{** WARNING: IEEEtran.bst: No hyphenation pattern has been}%
\typeout{** loaded for the language `#1'. Using the pattern for}%
\typeout{** the default language instead.}%
\else
\language=\csname l@#1\endcsname
\fi
#2}}
\providecommand{\BIBdecl}{\relax}
\BIBdecl

\bibitem{Dubins}
L.~E. Dubins, ``On curves of minimal length with a constraint on average
  curvature, and with prescribed initial and terminal positions and tangents,''
  \emph{American Journal of Mathematics}, vol.~79, 1957.

\bibitem{Reeds_Shepp}
J.~A. Reeds and L.~A. Shepp, ``Optimal paths for a car that goes both forwards
  and backwards,'' \emph{Pacific Journal OF Mathematics}, vol. 145, 1990.

\bibitem{PMP}
L.~S. Pontryagin, V.~G. Boltyanskii, R.~V. Gamkrelidze, and E.~F. Mishchenko,
  \emph{The mathematical theory of optimal processes}.\hskip 1em plus 0.5em
  minus 0.4em\relax Interscience Publishers, 1962.

\bibitem{sussman_geometric_examples}
H.~J. Sussmann and G.~Tang, ``Shortest paths for the reeds-shepp car: a worked
  out example of the use of geometric techniques in nonlinear optimal
  control,'' Rutgers University, Tech. Rep., 1991.

\bibitem{boissonat}
J.-D. Boissonnat, A.~Cerezo, and J.~Leblond, ``Shortest paths of bounded
  curvature in the plane,'' in \emph{Proceedings 1992 IEEE International
  Conference on Robotics and Automation}, 1992, pp. 2315--2320.

\bibitem{phase_portrait_kaya}
C.~Y. Kaya, ``Markov–dubins path via optimal control theory,''
  \emph{Computational Optimization and Applications}, vol.~68, pp. 714--747,
  2017.

\bibitem{Shortest_path_synthesis_Boissonat}
X.-N. Bui, J.-D. Boissonnat, P.~Soueres, and J.-P. Laumond, ``Shortest path
  synthesis for dubins non-holonomic robot,'' in \emph{Proceedings of the 1994
  IEEE International Conference on Robotics and Automation}, 1994, pp. 2--7.

\bibitem{dubins_classification}
A.~M. Shkel and V.~Lumelsky, ``Classification of the dubins set,''
  \emph{Robotics and Autonomous Systems}, vol.~34, pp. 179--202, 2001.

\bibitem{path_synthesis_reeds_shepp}
P.~Soueres and J.-P. Laumond, ``Shortest paths synthesis for a car-like
  robot,'' \emph{IEEE Transactions on Automatic Control}, vol.~41, no.~5, pp.
  672--688, 1996.

\bibitem{monroy}
F.~Monroy-Pérez, ``Non-euclidean dubins' problem,'' \emph{Journal of Dynamical
  and Control Systems}, vol.~4, pp. 249--272, 1998.

\bibitem{sussman_3D}
H.~Sussmann, ``Shortest 3-dimensional paths with a prescribed curvature
  bound,'' in \emph{IEEE Conference on Decision and Control}, 1995, pp.
  3306--3312.

\bibitem{dubins_circle}
S.~G. Manyam, D.~W. Casbeer, A.~V. Moll, and Z.~Fuchs, ``Shortest dubins paths
  to intercept a target moving on a circle,'' \emph{Journal of Guidance,
  Control, and Dynamics}, pp. 1--14, 2022, article in advance.

\bibitem{sinistral/dextral}
E.~Bakolas and P.~Tsiotras, ``The asymmetric sinistral/dextral markov-dubins
  problem,'' in \emph{Proceedings of the 48h IEEE Conference on Decision and
  Control (CDC)}, 2009, pp. 5649--5654.

\bibitem{PMP_lecture_notes}
P.~D. Loewen, ``The pontryagin maximum principle,'' 2012, {M403} Lecture Notes.

\end{thebibliography}

\end{document}